\documentclass[11pt,a4paper,twoside]{article}
\usepackage{amsthm,amsfonts,amsmath,amscd,amssymb}
\usepackage{latexsym}
\usepackage{euscript}
\usepackage{enumitem}
\usepackage{graphicx}
\usepackage{tikz}
\usepackage{caption}
\usepackage{subcaption}
\usepackage{cite}
\usepackage[utf8]{inputenc}
\usepackage[english]{babel}
\usepackage{xcolor}
\usepackage[hyperpageref]{backref}

\usepackage{jmpag} 

\DeclareMathOperator{\arccot}{arccot}
\DeclareMathOperator{\arccoth}{arccoth}

\newcommand{\Int}{\text {int} \ }

\newcommand{\e}{\varepsilon}

\newcommand{\F}{\mathcal F} 
 
\newcommand{\p}{\mathcal P} 

\newcommand{\R}{\mathbb R} 
\newcommand{\Z}{\mathbb Z} 

\begin{document}


\title{The $L^2$-norm of the Euler class for  Foliations on closed irreducible Riemannian 3-Manifolds}


\title[Euler class of  bounded mean curvature foliations on 3-Manifolds]{The $L^2$-norm of the Euler class for  Foliations on closed irreducible Riemannian 3-Manifolds}

\author{Dmitry V. Bolotov}

\address{B. Verkin Institute for Low Temperature Physics and Engineering of the
	National Academy of Sciences of Ukraine, 47 Nauky Ave., Kharkiv, 61103, Ukraine}
\email{bolotovo@ilt.kharkov.ua}
\orcid{0000-0002-8542-9695}




\BeginPaper 



\newcommand{\ep}{\varepsilon}
\newcommand{\eps}[1]{{#1}_{\varepsilon}}



\begin{abstract}
An upper bound for the $L^2$- norm of the Euler class $e(\F)$ of an arbitrary transversally orientable foliation $\F$ of codimension one, defined on a three-dimensional closed irreducible orientable Riemannian 3-manifold $M^3$ is given in terms of constants bounding the volume, the radius of injectivity, the sectional curvature of $M^3$ and the modulus  of mean   curvature of the leaves.  As a consequence we get that only finitely many cohomolo\-gical  classes of the group $H^2(M^3)$ that can be realized by the Euler class $e(\F)$ of a two-dimensional transversely oriented foliation $\F $ whose leaves  have the modulus of mean curvature which is bounded above by the fixed constant $H_0$. 
	
	\key{3-Manifold, foliation, Euler class,  mean curvature.}
	
\end{abstract}

\section{Introduction}

Let $ (M^3, g) $ be a closed oriented three-dimensional Riemannian manifold and  $ \F $ be a transversely oriented $ C ^{\infty} $-smooth foliation of codimension one on $ M^3 $. 
Recall that a foliation $\F $ is {\it taut}  if its leaves are minimal submanifolds of $ M^3 $ for some Riemannian metric on $ M^3 $. D. Sullivan \cite{Sul} gave a description of taut foliations, namely, he proved that a foliation is taut if and only if each leaf of $\F$ is intersected by a transversal closed curve, which   is equivalent to $\F$  does not contain generalized Reeb components (see bellow).

We previously proved the following result  \cite {B1}.

\begin{theorem} \label{result1}
	
	Let $ V_0> 0, i_0> 0, K_0 \geq 0 $ be fixed constants, and $ M^3 $ be a closed oriented three-dimensional Riemannian manifold with the following properties:
	\begin{enumerate}
		\item the volume $Vol(M^3)\leq V_0$; 
		\item the sectional curvature $K$  of  $M$  satisfies the  inequality $K  \leq K_0$;
		\item  $\min \{inj(M^3),\frac{\pi}{2\sqrt{K_0}}\} \geq i_0 $, where $ inj (M^3) $ is the injectivity   radius of  of  $ M^3 $.
	\end{enumerate}
	Let us set
	\[ H_0 = \begin{cases}
		\min \{{\frac {2\sqrt{3}i_0^2}{V_0},\sqrt[3]{\frac{2\sqrt{3}}{V_0}}}\},     & \text{if $K_0 = 0$,}\\
		\min \{ {\frac {2\sqrt{3}i_0^2}{V_0}, x_0 } \},  &  \text{if $K_0 >0$,}
	\end{cases} \]
	where $x_0$ is the root of the equation
	$$ \frac{1}{ K_0 }\arccot^2\frac{x}{\sqrt{K_0}} - \frac{V_0}{2\sqrt{3}}x=0.  $$
	Then any smooth transversely oriented foliation $ \F $ of codimension  one on $ M^3 $, such that the modulus of the mean curvature $H$ of its leaves  satisfies the inequality $ | H | <H_0 $,
	should be taut, in particular, have minimal leaves for some Riemannian metric on $ M^3 $.
	
\end{theorem}

Notice that    if $M^3$ admits a taut foliation, then $M^3$ is irreducible \cite{Nov}. Let us recall that 3-manifold $M^3$ is called {\it irreducible} if each an embedded sphere bounds  a ball in $M^3$. In particular,   $\pi_2(M^3)=0$ (see \cite{H}).  

W. Thurston has proved  \cite{Th} (see also \cite{ET}) that if $M^2\subset M^3$ is a closed embedded  orientable surface which is different from $S^2$ then  the Euler class $e(\mathcal F)$ of  a transversely oriented taut foliation  $ \mathcal F $ on $M^3$ satisfies\footnote  {By the Euler class of  the foliation $\F$ we mean the Euler class of the  distribution tangent to $\F$.}:

\begin{equation}\label{Chi}
	|e(\mathcal F)[M^2]|\leq - \chi (M^2). 
\end{equation}

Since any integer homology class $H_2(M^3;\Z)$ can be represented by a closed oriented surface (see subsection \ref{harm}) the inequality above bounds  the possible values of the cohomology class $e(\F)$ on the generators of $H_2(M^3;\Z)$ and  therefore  the number of cohomological classes $ H^2 (M^3;\Z) $ realized as Euler classes $e(\mathcal F)$ is finite.

In this paper we estimate  from above the $L^2$ -norm of the Euler class  of foliations on closed  Riemannian 3-Manifolds with leaves having a  mean curvature bounded in absolute value by some positive constant.  Here we represent the following result.

\begin{theorem}[Main theorem]\label{main}
	Let $ V_0> 0, i_0> 0, H_0>0, k_0 \leq K_0$ be fixed constants.  Suppose   $ (M^3,\F) $ be a closed oriented irreducible three-dimensional Riemannian manifold equipped by  a two-dimensional transversely oriented foliation $\F $, whose leaves  have the modulus of mean curvature $H$ bounded above by the constant $H_0$,  and $M^3$  satisfies the  following conditions:
	\begin{enumerate}
		\item the volume $Vol(M^3)\leq V_0$; 
		\item the sectional curvature $K$  of  $M$  satisfies the  inequality $k_0\leq K  \leq K_0$;
		\item  $$\left\{\begin{array}{lr} \min \{inj(M^3),\frac{\pi}{2\sqrt{K_0}}\}, &\text {if } K_0>0\\
			inj(M^3)&\text {if } K_0\leq0 	\end{array}  \right\}\geq i_0,$$ where $ inj (M^3) $ is the injectivity   radius of  of  $ M^3 $.
	\end{enumerate}
	Then 
	there exists a constant $C_1(V_0,i_0,k_0,K_0,H_0)$ such that $L^2$ -norm $$||e(\F)||_{L^2}\leq C_1.  $$
	 
\end{theorem}

\begin{corollary}
	For any closed oriented Riemannian 3-Manifold	$M^3$  there are only finitely many cohomological  classes of the group $H^2(M^3;\R)$ that can be realized by the Euler class $e(\F)$ of a two-dimensional transversely oriented foliation $\F $ whose leaves  have the modulus of mean curvature bounded above by the fixed constant $H_0$. 
\end{corollary}

\begin{remark} In the theorem \ref{main} the Euler class $e(\F)$ is assumed to be real, i.e. image of the integer Euler class via the homomorphism $H^2(M^3;\Z) \to H^2(M ^3;\R)$ induced by the embedding of the coefficients $\Z\hookrightarrow \R$.  Clearly, $e(\F)\in H^2(M^3;\Z) _{\R}\subset H^2(M ^3;\R)$, where $H^2(M^3;\Z) _{\R}$ is an integer lattice in $ H^2(M ^3;\R)$. Recall also that the real cohomology groups  are isomorphic to the de Rham cohomology groups and we can represent the real Euler class through a closed differential form, in particular, the harmonic form (see subsection \ref{harm}).   
\end{remark}
	\begin{remark} \label{1.4}
	 As follows from  Myers's Theorem \cite{My} if $k_0>0$, then $\pi_1(M^3)$ is finite and $H_1(M^3;\R) \cong H^2(M^3;\R)\equiv 0$,  which implies $e(\F)=0$. Thus we can suppose that $k_0\leq0$. 
\end{remark}

\begin{remark}
	The foliation $\F$ does not contain a sphere as  a leaf since  in this case, by  Reeb's stability theorem (see \cite{T}),    $M^3\simeq S^2\times S^1$, which   contradicts  the irreducibility of $M^3$.
\end{remark}

\section {Background material}
\subsection{Geometrical inequalities} 
\label{sec:1}

\subsubsection{Inequalities associated  with  a generalized Reeb component}\label{ineq}

A subset of the foliated manifold $ (M,\F) $ is called {\it a saturated set} if it consists of leaves of the  folation $ \F $. A saturated set $ A $ of a three-dimensional compact orientable manifold $ M^3 $ with a given transversely orientable foliation $ \F $ of codimension one is called a {\it generalized Reeb component} if $ A $ is a connected three-dimensional manifold with a boundary $\partial A$ and any transversal to $ \F $ vector field restricted to   $\partial A $ is directed either everywhere inwards or everywhere outwards of the generalized Reeb component $ A $. In particular, the Reeb component $ R $ (see \cite {T}) is a generalized Reeb component. It is clear that $ \partial {A} $ consists of a finite set of compact leaves of the foliation $ \F $.  It is not difficult to show that $ \partial A $ is a family of tori (see \cite {Good}).

	
	The next result is due to G. Reeb.
	
	\begin{theorem}\cite{R}\label{R}
		Let $ (M^3, g) $ be a closed oriented three-dimensional Riemannian manifold and $ \F $ be a smooth transversely oriented foliation of codimension one on $ M $. Then
		\begin{equation} \label{reeb}
			d \chi = 2H \mu, 
		\end{equation}
		where $ \chi $ is the  volume form of  the foliation $ \F $, and $ \mu $ is the  volume form on $ M^3 $.
	\end{theorem}

	\begin{corollary}\label{th1}

		Let $ M^3$ be a closed oriented three-dimensional Riemannian manifold with a given transversely oriented smooth foliation $ \F $ of codimension one. Suppose that $ \F $ contains a generalized Reeb component $ A $ and the modulus of the mean curvature $H$ of the foliation $ \F $ is bounded above by   $ | H | \leq H_0 $. Then
		\begin{equation}\label{ar}
			Area  (\partial A) \leq 2H_0 Vol(A) \ \& \ \   
			Area  (\partial A) \leq H_0 Vol(M^3)
		\end{equation}
	\end{corollary}
	\begin{proof}
		$$
		0<Area (\partial A)= |\int_{\partial A}\chi| \overset{(Stokes)}= |\int_Ad\chi| \overset{(\eqref{reeb})}=|2\int_{A} H\mu | \leq 2\int_{A} H_0\mu = 2H_0 Vol(A). 
		$$
		Let $B=M^3\setminus \Int A$. Then $B$ is also generalized Reeb component and   we have:
		
		$$
		Area (\partial B =\partial A)= |\int_{\partial B}\chi|= |\int_Bd\chi| =|2\int_{B} H\mu | \leq 2\int_{B} H_0\mu = 2H_0 Vol(B).
		$$

		It follows that 
		
		$$ 2Area (\partial A) \leq 2H_0(Vol (A) + Vol (B)) \leq 2H_0 Vol (M^3) ,$$  which implies the result.

	\end{proof}

	\begin{corollary}
		The generalized Reeb component $ A $ is an obstruction to the foliation $ \F $ being taut.  
	\end{corollary}
	\begin{remark}
		The converse is also true, if the foliation is not taut, then it contains a generalized Reeb component (see   \cite{Good}).
	\end{remark}
	
	\subsubsection{Systolic inequalities}
	Recall that the  systole,  denoted by \emph{sys},  in a Riemannian  manifold $M$ with non-trivial fundamental group  is the length of the smallest loop in $M$ that is not null-homotopic  in $M$.   Under the condition of closeness $M$  such a loop exists and is necessary  a closed geodesic.     The proof does not differ from the proof of the existence of a closed geodesic in its free homotopy class (see \cite[Chapter 12]{DC}).

	The Loewner theorem below gives an upper bound on the systole  in a Riemannian two-dimensional torus.
	
	\begin{theorem}\label{sys1} (Loewner)\cite{Pu}
		Let $T^2$ be a two-dimensional torus with an arbitrary Riemannian metric  on it.
		Denote by $sys$ (abbreviated from {\it systole}) the length of the shortest closed noncontractible geodesic on $T^2$.  Then
		\begin{equation}\label{sys2}
			sys^2 \leq \frac {2} {\sqrt{3}} Area (T^2).   
		\end{equation}
	\end{theorem}

	Due to Gromov, the generalization of this theorem is the following:
	
	\begin{theorem} (see \cite[ Chapter 6]{Katz})
		Let $T^2$ be a two-dimensional torus with an arbitrary Riemannian metric  on it. Then there exists a pair of closed geodesics on $T^2$ of respective length $\lambda_1$, $\lambda_2$, such that  
		\begin{equation}\label{sys3}
			\lambda_1\lambda_2 \leq \frac {2} {\sqrt{3}} Area (T^2),   
		\end{equation}
		and whose homotopy classes		form a generating set of $\pi_1(T^2)= \Z^2$.
	\end{theorem}

	\begin{corollary}\label{prop1}
		Let $T^2$ be a Riemannian torus for which $$sys \geq C_0 , \  Area (T^2)\leq S_0.$$ for some positive constants $C_0, S_0$.  Then there exists a pair of closed geodesics on $T^2$ whose homotopy classes		form a generating set of $\pi_1(T^2)= \Z^2$ and whose  lengths  $\lambda_1$, $\lambda_2$  do not exceed some constant $C(C_0,S_0)$.
		
	\end{corollary}
	\begin{proof}
		From \eqref{sys3} immediately follows that 
		\begin{equation}\label{C}
			\lambda_i\leq  \frac {2} {\sqrt{3}} \frac {Area (T^2)}{sys}\leq C:=\frac{2S_0}{\sqrt{3}C_0}, \ i=1,2.
		\end{equation}
	\end{proof}

	The concept of systole can be generalized to foliations. 
	
	\begin{definition}
		Let $(M,\F)$ is foliated manifold.  Following   \cite[Chapter VII]{HH} we will call a  loop  $f:S^1\to M$ {\it integral}  for $\F$  if  $f(S^1)$ is contained in some leaf $\mathcal L$ of $\F$.  In this case $\mathcal L$ is referred as the {\it support} of $f$. 
	\end{definition}
	\begin{definition}\label{essen}
		The integral loop    supported by  $\mathcal L$ are refereed to as {\it essential} if  the loop $f:S^1\to \mathcal L$ represents nontrivial element  of the fundamental group $\pi_1(\mathcal L)$ and  {\it  inessential} otherwise.
	\end{definition}

	We recently proved the following theorem.
	
	\begin{theorem}\cite{B2}\label{defsys}
		Let $(M,\F)$ be a foliated closed Riemannian manifold containing a leaf with a nontrivial fundamental group.  Then there is  an integral  essential  loop $l_{sys}$ in $M$ with smallest length  among all integral essential loops in $(M,\F)$, which is necessary a closed geodesic in its support.    
	\end{theorem}
	
	\begin{definition}
		Denote by $sys(\F)$ the length of the  geodesic $l_{sys}$ from Proposition \ref{defsys}. 
		\newline
	\end{definition}

	\subsubsection {Comparison inequalities}
	
	Recall the following  comparison theorem for the normal curvatures.
	
	\begin{theorem}\cite [ 22.3.2.]{BZ}\label {BZ}
		Let $p\in M$ and $\ \beta:[0,r]\to M$ be a radial  geodesic  of the ball $B(p,r)$ of radius $r$ centered at the point $p$ of the Riemannian manifold $M$. Let   $\beta(r)$ be a point  not conjugate with $p$ along $\beta$. Let the radius $r$ be such that there are no conjugate points in the space of constant curvature $K_0$ within the radius of length $r$. Then if at each point $\beta (t)$ the sectional curvatures $K$ of the manifold $M$ do not exceed $K_0$, then the normal curvature $k^S_n$ of the sphere $S(p,r)$ at the point $\beta(r)$ with respect to the normal $-\beta'$ is not less than the normal curvature $k_n^0$ of the sphere of radius $r$ in the space of constant curvature $K_0$.
	\end{theorem}

	Let $M^3$ be a 3-Manifold satisfying the condition of  Theorem \ref{main}.  Note that   all normal curvatures of the sphere $S(r)\subset M^3$ of radius $r$ are positive, provided that $r<i_0$ and the normal to the sphere $S(r)$ is directed inside the ball $B(r)$ which it bounds\footnote {The sphere $S(r)$ indeed bounds the ball, since by definition $r< inj (M^3)$.}. We will call such a normal  {\it inward}.

	\begin{definition}
		We  will call a hypersurface $S\subset M^3$ of the Riemannian manifold $M^3$    {\it supporting} to the subset $A\subset M^3$ at the point $p\in \partial A\cap S$ with respect to the normal $n_p \perp T_pS$, if $S$ cuts some spherical neighborhood $B_p$ of the point $p$ into two components, and $A\cap B_p$ is contained in that component to which the normal $n_p$ is directed.
		We will call the sphere $S(r)\subset M^3$ ($r< i_0$) the {\it supporting sphere} to the set $A\subset M^3$ at the point $q\in A\cap S(r)$ if it is the supporting sphere to $A $ at the point $q$ with respect to the inward normal.
	\end{definition}

	The following  lemma is obvious. 
	
	\begin{lemma}(\cite[Lemma 4]{B1})\label{l1}
		Assume that the sphere $S(r_0)$ ($r_0< i_0$) is the supporting sphere to the surface $F\subset M^3$ at the point $q $. Then $k^S_n(v) \leq k_n^F(v)\ \forall v\in T_qS(r_0)$, where $k^S_n(v)$ and $k^F_n(v)$ denote corresponding normal curvatures of $S(r_0)$ and $F$ at the point $q$ in the direction $v$.
	\end{lemma}
	
	As a consequence of   Lemma \ref{l1} and  Theorem \ref{BZ}  we obtain the following inequalities  at the touching point $q$ :
	\begin{equation}\label{ineq1}
		0< H_r^0\leq H_r(q) \leq H (q),
	\end{equation}
	where $H^0_r $ and $H_r$ are the mean curvatures of the spheres $S(r)$ bounding the ball of radius $r, \ r<i_0,$ in the space of constant curvature $K_0$ and the manifold $M^3$ respectively, and $H$ is the mean curvature of the surface $F$ .

	\subsection{Harmonic maps to the circle and harmonic forms.}\label{harm}
	
	Let $M^3$ be a closed oriented Riemannian 3-Manifold. 
	Recall that 
	\begin{equation}\label{S^1}
		H^1(M^3)\cong [M^3,S^1]
	\end{equation} 
	and each cohomological class  $a\in H^1(M^3;\Z)$  can be obtained as an  image of the generator $[S^1]^*\in H^1(S^1;\Z)\cong \Z$ under the homomorphism $f^*:H^1 (S^1)\to H^1(M^3;\Z)$ induced by  the  mapping $f: M^3\to S^1$  uniquely  defined up to homotopy. Recall that the group $H_2(M^3;\Z)\overset{PD}\cong H^1(M^3;\Z)$ does not contain a torsion and we can identify $H^1(M^3;\Z)$ with  integer lattice $H^1(M^3;\Z)_{\R}\subset H^1(M^3;\R)$ and $H_2(M^3;\Z)$ with  $H_2(M^3;\Z)_{\R}\subset H_2(M^3;\R)$ respectively. Observe that  the Poincar\'e duality $H^1(M^3;\R)\overset{PD}\cong  H_2(M^3;\R)$ induces the Poincar\'e duality  of integer lattices  $H^1(M^3;\Z)_{\R}\overset{PD}\cong  H_2(M^3;\Z)_{\R}$. 
	
	Let  us identify $S^1 $ with the unit  length circle $\R/\Z$ with the natural parameter $\theta$. If $f$ is  a smooth function, then the preimage $f^{-1}(\theta)$ of a regular value $\theta\in S^1$ is a smooth (not necessarily connected)  oriented    submanifold $M^2\subset M^3$, which we identify with the image of the  embedding  $i: M^2 \hookrightarrow M^3$.    The singular homology class $[M^2,i]:=i_*[M^2]\in H_2(M^3;\Z)_{\R}$ corresponding to the singular cycle  $(M^2, i)$  is     Poincar\'e dual to the cohomology class   $a\in H^1(M^3;\Z)_{\R}$, where $[M^2]\in H_2(M^2;\R)$ denotes a fundamental class of $M^2$, which is  generator of the group $\Z\cong H_2(M^2;\Z)_{\R}\subset  H_2(M^2;\R)\cong \R$. 
	
	\begin{remark}\label{r2}
		Note  that  by Sard's theorem the set of regular values of $f$ has a full  measure   in $S^1$  and    is  also an open set in $S^1$  since $M^3$  is compact. The same is true for any  smooth map $g:N\to L$ of smooth compact manifolds $N$ and $L$ \cite{PM}.   
	\end{remark}

	Now should recall that each homotopy  class in $ [M^3,S^1]$ can be represented by the harmonic mapping (\cite{ES}). 	Let $u:M^3\to S^1$ be  a harmonic map representing  the nontrivial  class $[u] \in [M^3,S^1]\cong H^1(M^3;\Z)$.  Observe that $\alpha = u^* d\theta, \ \theta \in S^1,$ is a harmonic 1-form (i.e. $d\alpha=\delta\alpha=0$) on $M^3$ corresponding to the  integer lattice class $[u]\in H^1(M^3;\Z)_{\R}$. 
	
	On the space of differential $k$ -forms $\Omega^k(M^3)$, $k\in \{0,1,2,3\}$,   one can introduce  $L^2$-norm:  
	\begin{equation}\label{hn} 
		||\alpha||_{L^2} =\sqrt{\int_{M^3}\alpha \wedge *\alpha} = \sqrt{\int_{M^3}|\alpha|^2},\end{equation}
	where $*$ denotes the Hodge star operator  and  $|\alpha_p| = \sqrt{*(\alpha_p \wedge *\alpha_p)},$  $\ p\in M^3.$
	Observe,  in 3-dimensional vector space $T_pM^3$ each  k-form  $\alpha_p$  is simple and   $|\alpha_p|$ coincides with the comass norm: 
	$$  |\alpha_p| = \max \alpha_p(e_1,\dots,e_k),$$
	where the maximum  is taken over all orthogonal  frames of vectors $(e_1,\dots,e_k)$ in $T_pM^3$. 
	
	We will also use the   $L^{\infty}$-norm on $\Omega^*(M^3)$, which is defined as follows:
	$$||\alpha||_{L^{\infty}}=\max_{p\in M^3} |\alpha_p|.$$
	
	The norm \ref{hn} induces the $L^2$-norm on the de Rham cohomology of $M^3$ as follows. Let $a\in H^k(M^3;\R)$, then we set 
	
	$$ ||a||_{L^2}:= \inf_{\alpha} \{||\alpha||_{L^2}: \alpha \in \Omega^k(M^3) \  \text{is a smooth closed k-form representing }  a\}. $$
	
	From de Rham  - Hodge theory it follows  $||a||_{L^2} = ||\alpha||_{L^2} $, where $\alpha $ is the unique harmonic form ($d\alpha=\delta\alpha=0$) representing the class  $a\in H^k(M^3;\R)$.

	Using Poincar\'e duality $H_i(M^3;\R)\overset {PD}\cong H^{3-i}(M^3;\R)$ we can introduce the $L^2$ - norm on $H_2(M^3;\R)$ setting $$||b||_{L^2} = ||PD(b)||_{L^2}, \ b\in H_i(M^3;\R). $$  
	On the other hand, the non-degenerate Kronecker pairing
	$$
	<,>:H^k(M^3;\R) \times H_k(M^3;\R) \to \R
	$$
	induced by integration of closed forms over cycles, allows us to define  the $L^2$- norm  $||\cdot||^*_{L^2}$  on $H_k(M^3;\R)\cong (H^k(M^3;\R))^*$  dual to the $L^2$ - norm  $|| \cdot||_{L^2}$ on $H^k(M^3;\R)$. 
	As was shown in \cite {BK}  
	$$
	PD: (H^i(M^3 ;\R), || \cdot||_{L^2}) \to (H_{3-i}(M^3; \R), || \cdot||^*_{L^2}) 
	$$
	is an isometry for $i=1,2$.

	Note,  that $$PD([\alpha\wedge \beta] )=  PD([\beta\wedge \alpha])=<[\alpha],PD([\beta])>= <[\beta],PD([\alpha])>,$$
	where $\alpha\in \Omega^1(M^3)$ and $\beta\in \Omega^2(M^3)$ are closed forms. Since the  set of  integer-directed rays   from   $0\in H^1(M^3;\R)$  are everywhere dense set in  $H^1(M^3;\R)$.    we have:
	
	\begin{equation}\label{*-norm}
		|| b ||_{L^2} = ||PD(b)||^*_{L^2} = \sup_{a\not=0}\frac {<a,PD(b)>}{||a||_{L^2}}= \sup_{[\Sigma]\not = 0}\frac {<b,[\Sigma]>}{||[\Sigma]||_{L^2}}
	\end{equation}   
	where  $b \in H^2(M^3,\mathbb R)$,  $a \in H^1(M^3,\mathbb \Z)_{\R}$ and  $\Sigma$ is a  compact oriented surface embedded in $M^3$ such that  $PD(a)=[\Sigma]. $

	Let us recall the following    inequality   (see \cite[[7.1.13,7.1.17, 9.2.7,9.2.8]{Pet}).  If $\alpha$ is a harmonic 1-form on closed Riemannian manifold $M^n$,  then 
	
	\begin{equation}\label{Pet}
		||\alpha||_{L^{\infty}} \leq \Lambda_n(k,D) ||\alpha||_{2}.
	\end{equation} 
Here $||\alpha||_{2} =\frac{||\alpha||_{L^{2}}}{\sqrt{Vol(M^n)}}$,   $D>0$ is the constant satisfying the inequality   $Diam(M^n)\leq D$  and  $k\leq 0$   is the constant satisfying the inequality $Ric(M^3)\geq (n-1)k$.

	\begin{remark}\label{Croke}
				C.B. Croke   in \cite{Croke}   gave an estimate for the diameter of a closed  Riemannian manifold,  which we  adapt to 3-dimensional case:
		$$
		Diam (M^3)\leq  \frac{27\pi Vol(M^3)}{8\ inj(M^3)^{2}}.
		$$  
		In particular, if $M^3$ satisfies the conditions of  Theorem \ref{main} we can take  $$D=\frac{27}{8}\pi\frac{ V_0}{ i_0^{2}}.$$
		Moreover, we can put  $k=k_0$ (see Remark \ref{1.4}) and thus 
		in our case we have 
		
		\begin{equation}\label{V}
		\Lambda_3(k,D) = \Lambda(V_0, i_0,k_0).
		\end{equation}
	\end{remark}

	The following Stern's theorem estimates an average Euler characteristic of a surface dual to the harmonic mapping of $M^3$  into the circle.
	
	\begin{theorem}\cite{St}\label{St}
		Let $u:M^3\to S^1$ be  a harmonic map to the unit length circle  representing  the nontrivial  class $[u] \in [M^3,S^1]\cong H^1(M^3 ;\Z)\overset{PD}\cong H_2(M^3;\Z)$. Then 
		\begin{equation}\label{St0}
			2\pi \int_{\theta\in S^1}\chi(\Sigma_{\theta})\geq \frac{1}{2} \int_{\theta\in S^1}\int_{\Sigma_{\theta}}(|du|^{-2}|Hess(u)|^2 + R_{M^3}),
		\end{equation}
		where $\Sigma_{\theta}=u^{-1}{\theta}$,  $\theta \in S^1$, and  $R_{M^3}$ is the scalar curvature of $M^3$.
	\end{theorem}
	
	\begin{remark}\label{comp} 
		For  a regular  value $\theta \in S^1$ of $u:M^3 \to S^1$ each  connected component $\Sigma^i_{\theta}$ of $\Sigma_{\theta}$  represent a non-trivial homology class in $H_2(M^3)$  (see \cite{St}), and  since $M^3$ is assumed to be irreducible,  $\chi(\Sigma^i_{\theta})\leq0$. 
	\end{remark}
	
	As a corollary of \eqref{St0},  Remark \ref{comp} and  Cauchy–Schwarz  inequality we have  the  following useful estimate.
	
	\begin{theorem}\cite{St}
		\begin{equation} \label{St1}
			\int_{\theta\in S^1}\chi(\Sigma_{\theta}) \geq - \frac{1}{4\pi}||\alpha||_{L^2}||R^-||_{L^2},
		\end{equation}
		where $R^-:=\min\{0,R\}$ is a non-positive  part of the scalar curvature $R$ and $\alpha = u^*d\theta$.
	\end{theorem}

	\subsection{Novikov's theorem and  vanishing cycle.} \label{2.3}

	Let $(M^3,\F)$ be a foliated closed  3- Manifold. An integral  loop   $\alpha:S^1\to M^3$ is a {\it vanishing cycle} if there exists a homotopy $A: S^1\times I \to M^3$ through integral loops $A_t:=A|_{S^1\times t}$ for $\F$ such that $A_0=\alpha$ and $A_t$ is inessential for $0<t\leq 1$. 
	A vanishing cycle $\alpha $ is {\it non-trivial} if $\alpha$ is essential.
	
	The following well-known Novikov's theorem gives topological obstruction to the existence of taut foliations.

	\begin{theorem}  \cite {Nov}\label{Nov}. 
		\begin{enumerate}
			\item For a closed, orientable smooth 3-Manifold $M^3$ and a transversely orientable  $C^2$-smooth foliation $\F$ of codimension one on   $M^3$, the following are equivalent.
			\begin{enumerate}
				\item  The foliation $\F$  has a Reeb component.
				\item There is a leaf $L$ of $\F$  that is not $\pi_1$-injective. That is, the inclusion $i: L \to M^3$ induces a homomorphism $i_*:\pi_1({ L})\to \pi_1(M^3)$  with nontrivial kernel. 
				\item Some leaf of $\F$  contains a non-trivial vanishing cycle.
			\end{enumerate}
			\item{The  support of the non-trivial vanishing cycle is a torus  bounding a Reeb component.}
		\end{enumerate}
	\end{theorem}

	The construction underlying the proof of Novikov's theorem is as follows. Let a simple closed integral regular curve $\alpha: S^1 \to   M^3$ belongs  to the leaf ${ L}\in \mathcal F$ and represents the non-trival element of $Ker (i_*: \pi_1({ L})\to \pi_1(M^3))$.  We can  find  an immersion $p: D \to M^3$  of the  two-dimensional disk $D$ such that $p (\partial D) = \alpha$. 
	This immersion can be brought to a general position by a small perturbation (modulo $\partial D$). It means that the induced foliation $\F':=p^{-1}(\F\cap p(D))$ has only Morse singularities (saddles and centers).  Moreover, by a small perturbation  we  can obtain  not more than one singular point on a single leaf (see \cite[Lemma 9.2.1.]{CC2}).  The resulting foliation outside the singular points on $D$ can be oriented (see subsection \ref{3.1.}). Therefore there is a smooth vector field $X$ tangent to $\mathcal F'$ with zeros  corresponding to the singular points of $\mathcal F'$. 
	Recall that a separatrix coming out  of a singular point and returning  to it, together with the singular point (a saddle) is called  {\it a  separatrix loop}. By the construction,  a saddle singular point of $\F'$ belongs to at most two separatrix loops.  
	
	The general position idea described above can be extended  to arbitrary immersed  compact surfaces. In particular, the following theorem holds.
	
	\begin{theorem}  \cite [{Theorem 7.1.10}]{CC1},\cite[9.2.A]{CC2}  \label{CC1} Let $M^3$ be an oriented  closed 3-manifold with a smooth  transversally oriented foliation $\F$ on it.  Then for any $C^q$-mapping $f:N^2\to M^3$ of a compact oriented  surface $N^2$ such that   in the case  of $\partial N^2\not =\emptyset$ we have  $f|_{\partial N^2}$  either is transverse to $\F$  or has image in a leaf $L$ of $\F$, and  for any $\delta> 0$ there exists a $\delta$-close to $f$ $C^q$ -immersion $p: N^2\to M^3$ in $C^q(N^2,M^3)$ -topology, $q\geq 2$, such that:
		
		\begin{enumerate}
			\item [I.]The induced foliation $\F':=p^{-1}(\F\cap p(N^2))$ has only Morse singularities;
			\item [II.]There is at most one singular  point on  one  leaf;
			\item [III.] In the case of $\partial N^2\not =\emptyset$ we have  $p|_{\partial N^2}$  either is transverse to $\F$  or has image in a leaf $L$ of $\F$.
			
		\end{enumerate}
		
	\end{theorem}
	
	An  immersion $p$ satisfying the properties $I, \ II, \ III$ of  Theorem \ref{CC1}  will be  referred to as  an immersion of general position. 
	
	\begin{definition}
		Let us  identify   closed orbits and separatrix loops  of $\F'$  with the  images of  corresponding  loops $f:S^1\to N^2$   which bypass them once along the trajectories of the vector field  $X$.    The loop  $f:S^1\to N^2$  are refereed to as {\it essential} if  the integral loop $p\circ f$ is  an essential and  {\it  inessential} otherwise.  Note, that due to Reeb stability theorem   inessential closed orbits have a "good neighborhood", i.e.  neighborhood consisting of inessential closed orbits.
	\end{definition}
	
	\begin{definition}\label{VC}
		Let $p:N^2\to M^3$ be an immersion of general position described above. Let us denote by  $\p$ a subset of $ N^2$, which is   topologically a disk with a boundary that is  either a  closed orbit  or a separatrix loop of $\F'$, or it is a pinched annulus (see Fig.\ref{ris2})  consisting of two separatrix loops with a common saddle point.  Suppose  that $ \partial \p$ has a "good collar"  in $\p$, i.e. a collar consisting of inessential closed orbits of $\F'$.   Clearly, $p$-image of  $ \partial \p $  will  represent  a  vanishing cycle. We will refer to $ \cal O:= \partial \p $   as a vanishing cycle too.   
	\end{definition}
	
	One of S.P. Novikov's key observations in \cite{Nov} was the proof of the existence   of  a non-trivial vanishing cycle  $ \cal O $ inside of $(D,\F')$ (see above).  
	
	\begin{figure}[h]
		\center{\includegraphics[scale=0.2]{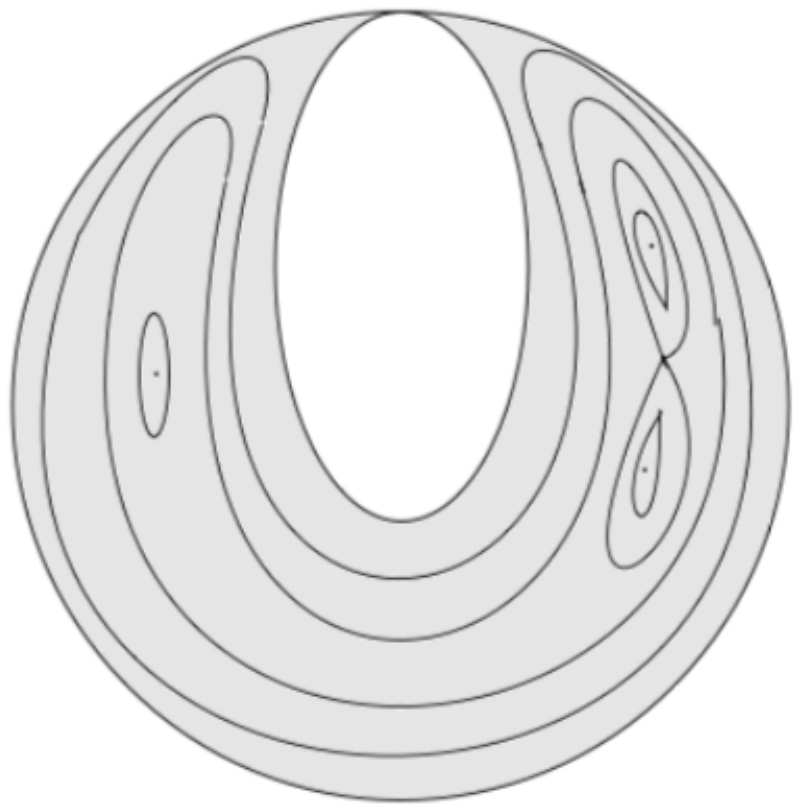}}
		\caption{Pinched annulus $\p$.}
		\label{ris2}
	\end{figure}
	
	\subsection {Euler class of  foliations.}\label{3.1.}
Here we describe  Thurston's construction for calculating the Euler class $e(\F)$ of a transversally oriented codimension one foliation $\F$ on a  closed oriented 3-Manifold $M^3$ (\cite{Th}).   
Let $p: N^2\to (M^3,\F)$ be an immersion of general position of a closed oriented surface $N^2$.  The  induced  foliation  $\F' =p^{-1}(\F\cap p(N^2) )$ on $N^2$ can be oriented outside the singular points. To verify this let us take a  normal vector field $n$ to the foliation $\F$ and    for all $x=p(z)\in p(N^2)$ consider  the orthogonal projection $n'(x)$ of the  normal   $n(x)$  to $\F$ onto the tangent plane $p_*(T_z (N^2))$, which in the case where $z$ is not a singular point uniquely determines the unit tangent vector $e'$ to the leaf $\mathcal L'_z\in \F', \ z\in \mathcal L'_z$, such that the frame $\{e',p_*^{-1}\frac {n'}{| n'|}\}$ defines a positive orientation of $T_z(N^2)$. Now we can define  a smooth vector field $X$ on $N^2$ tangent to $\F'$  whose zeros correspond to the singular points of $\F'$  putting
\begin{equation}\label{X}
	X=|n'|e'.
\end{equation}
\begin{remark}\label{perp}
	It is easy to define a vector field $X^{\perp}$ orthogonal to $\F'$ with respect to the induced  Riemannian metric on $N^2$. The vector field $X^{\perp}$  has the same singular points as $X$ and  the integral curves of $X^{\perp}$  define a foliation $\F'^{\perp}$ orthogonal to $\F'$ on $N^2$.
\end{remark}

The pair  $(N^2, p)$ can be understood as a singular cycle  if we fix some triangulation on $N^2$.   Let the singular homology class $[N^2,p]:=p_*[N^2]\in H_2(M^3;\Z)_{\R}\subset  H_2(M^3;\R)$ corresponds to the singular cycle  $(N^2, p)$, where $[N^2]$ denotes a fundamental class of $N^2$.  As W. Thurston showed \cite{Th}, to calculate the value of the Euler class $e(T\F)\in H^2(M^3,\Z)_{\R}$ of the foliation $\F$ on the singular homology class  $[N^2,p]\in H_2(M^3;\Z)_{\R}$, it suffices to calculate the total index of singular points of the vector field $X$ on $N^2$ taking into account the orientation of $p_*(T_q(N^2))$ at singular points \footnote{We apply Thurston's results to immersed submanifolds rather than embedded ones, where the same ideas work automatically.}. Since $M^3$ is oriented  we can uniquely choose  a unit normal vector $m\in T_{p(q)} M^3$  to the plane $ p_*(T_q(N^2), \ q\in N^2 $, which    defines the orientation of  $ p_*(T_q(N^2))$ coming from the orientation of $T_q(N^2)$.

We say that a singular point $q \in N^2$ is of {\it negative} type if $m(p(q))=-n(p(q))$. In the case when  $m(p(q))=n(p(q))$  the type of the singular point is called {\it positive}.

We denote by $I_N$ the sum of indices of  negative type singular points of the vector field $X$  and by $I_P$ the sum  of indices of  positive type singular points.  
The value of the Euler class $e(T\F)$ on the singular homology class $[N^2,p]$ is calculated as follows:
\begin{equation}\label{index}
	e(T\F)([N^2,p]) = e (p^*(T\F))([N^2]) =I_P -I_N.
\end{equation}

Recall that the Poincaré-Hopf theorem states that 

\begin{equation}\label{PH}
	\chi(N^2) =  I_P + I_N.
\end{equation}

	\section{Preliminary results}
	\subsection{ An upper bound for the number of Reeb components of a bounded mean curvature foliation} 
	
	The results of this subsections are represented in the article \cite{B2}.  For the sake of completeness, we present them in a slightly  more general form.

	Let us prove the following theorem.

	\begin{theorem}\label{rc}
		Let $ M^3 $ be a closed oriented three-dimensional Riemannian manifold satisfying the conditions $(1)-(3)$ of Theorem \ref{main}. 
		Let $\F$ be a codimension one transversally oriented foliation on $M^3$, whose leaves have a modulus of mean curvature bounded above by the fixed constant $H_0$.   
		
		Then \begin{equation}\label{C_0}			
			sys(\F)		\geq  C_0:= \left\{\begin{array}{lr}
				2 \min \{{i_0}, \frac{1}{ \sqrt{K_0}} \arccot\frac{H_0}{\sqrt{K_0}}\}  , &\text {if } K_0>0 \\
				2 \min \{i_0, \frac{1}{H_0}\}, &\text {if } K_0=0 \\
				2 \min \{{i_0}, \frac{1}{ \sqrt{-K_0}} \arccoth\frac{H_0}{\sqrt{-K_0}}\}  , &\text {if } K_0<0 \ \& \  H_0 >\sqrt{-K_0} \\
				2  {i_0}, &\text {if } K_0<0 \ \& \  H_0 \leq \sqrt{-K_0}

			\end{array}  \right\}. 
		\end{equation}
	\end{theorem}
	\begin{proof}
		\
		
		{\it Case 1.} $\frac{sys(\F)}{2} \geq i_0$. 	The result follows immediately.
		
		{\it Case 2.}  $\frac{sys(\F)}{2} < i_0$.

		Let $l_{sys}$ be an integral   closed geodesic  which is not null-homotopic in its support and whose  length  $sys =sys(\F) < 2i_0 $.  Then there is an immersion 
		$$p:D\to \Int B(r), \ r\in  (\frac{sys}{2}, i_0 )$$
		of a  disk   $D$ which is  in general position with respect to $\F$   and such that   $p(\partial D) =l_{sys}$.  	 As  noted in the section \ref{2.3}  there is a vanishing cycle    which  belongs to  $$T^2\cap p(D)\subset \Int B(r),$$ where  $T^2\in\F $ is a torus bounding   a Reeb component  $R$.

		Let  $r\in  (\frac{sys}{2}, i_0 )$ be a regular value of  the mapping
		\begin{equation}\label{Pr}
			pr_r|_{(\Int B(i_0)) \cap T^2 }: (\Int B(i_0)) \cap T^2 \to \R
		\end{equation}
		such that $pr_r(r,\phi_1, \phi_2) = r$,  where $(r,\phi_1, \phi_2)$ is a normal coordinate system in the ball $B(i_0)$.    
		
		In the case   $S(r)\cap  T^2\not =\emptyset$  	from  \cite[Proposition 2]{B1} follows that the sphere   $S(r)$ will be the a supporting sphere  with respect to the inward normal  at the tangent point $q$ for  some inner leaf of the Reeb component $ R$. 
		
		Note, due to     Sard's theorem   the set of regular values of the mapping \eqref{Pr}  has a full  measure   in the interval $(\frac{sys}{2}, i_0 )$ and the value $r$ can be taken arbitrarily close to $\frac{sys}{2}$.

		In the case of $S(r)\cap  T^2 =\emptyset$  we will achieve the tangency  of the sphere $S(r)$ and $T^2$  by decreasing the radius $r$, and the sphere $S(r)$ will become  supporting  for  the torus   $T^2$.  
		
		From \eqref{ineq1} it follows:
			\[					H_r^0 = \left\{\begin{array}{lr}
				\sqrt{K_0}  \cot  (r\sqrt{K_0}) , &\text {if } K_0>0 \\
				\frac{1}{r}, &\text {if } K_0=0   \\
				\sqrt{-K_0}  \coth  (r\sqrt{-K_0}),  &\text {if } K_0<0
				
			\end{array}  \right\} \leq H_0	\] 
		Observe  that $H_0$ must satisfy $	\sqrt{-K_0}< H_0$ if $K_0<0$.
		
		Thus   we conclude that  $sys(\F) $   must satisfy the inequality:

		\[				sys(\F)\geq   \left\{\begin{array}{lr}
			\ \frac{2}{ \sqrt{K_0}} \arccot\frac{H_0}{\sqrt{K_0}} , &\text {if } K_0>0 \\
			\ \frac{2}{H_0}, &\text {if } K_0=0  \\
			\frac{2}{ \sqrt{-K_0}} \arccoth\frac{H_0}{\sqrt{-K_0}} , &\text {if } K_0<0
		\end{array}\right\}	\]

		Combining Case $1$ and Case $2$ we obtain the result. 
	\end{proof}

	From Theorem \ref{rc} it follows:   
	
	\begin{corollary}\label{number}
		
		The number of Reeb components of the foliation $\F$  does not exceed $ \frac{4 H_0Vol(M^3)} {\sqrt{3}C^2_0}$.
	\end{corollary}
	
	\begin{proof}

		From Theorem \ref{sys1}  and  Corollary \ref{th1} we have
		\begin{equation}\label{ineq2}
			\frac{\sqrt{3}}{2}C_0^2\leq Area (\partial R) \leq 2H_0 Vol(R).
		\end{equation}

		From \eqref{ineq2} it follows
		$
		Vol(R)\geq  \frac{\sqrt{3}C^2_0}{4H_0}.
		$
		Since the interiors of Reeb components do not intersect,   the number of  Reeb components does not exceed $ \frac{4 H_0Vol(M^3)} {\sqrt{3}C^2_0}$.

	\end{proof}

	\subsection{Choosing a regular value of the harmonic mapping $$u:M^3\to S^1$$}

		Let $M^3$ from Theorem \ref{main} and $u:M^3\to S^1$ be  a harmonic map to the unit length circle representing  the nontrivial class $[u] \in [M^3,S^1]\cong H^1(M^3;\Z)$. 
	 Let us denote 
		\begin{equation} 
		{\bf A} = \{\theta \in S^1| -\chi(\Sigma_{\theta})\leq \frac{1}{2\pi}||\alpha||_{L^2}||R^-||_{L^2} \}
	\end{equation}
	where  $\alpha = u^*d\theta$ and $\Sigma_{\theta}=u^{-1}{\theta}$, $\theta \in S^1$ and $R^-$ is a non-positive part of the scalar curvature $R$.
	\begin{lemma}\label {L1}
	$ \mu({\bf A} )>\frac{1}{2}$ , 	where  $\mu$ denotes  standard Lebesgue measure.  
	
\end{lemma}
\begin{proof}
	If we assume that the statement of Lemma \ref{L1} is not true, then  taking into account Remark \ref{comp} we get 	$$ \mu(\theta \in S^1| -\chi(\Sigma_{\theta})> \frac{1}{2\pi}||\alpha||_{L^2}||R^-||_{L^2})\geq\frac{1}{2}$$
	and $$ \int_{\theta\in S^1}-\chi(\Sigma_{\theta}) > \frac{1}{4\pi}||\alpha||_{L^2}||R^-||_{L^2}$$
	that 	 contradicts to \eqref{St1}. 
\end{proof}

It follows from  Corollaries \ref{prop1} and   \ref{th1} that  every torus $T^2_j $ bounding the Reeb component of ${ R}_j\in \F$ contains  a simple closed smooth curve $\gamma_j$ which is non-homologous to zero in ${ R}_j$ and  has a length bounded above by the constant $C = \frac {2H_0Vol(M^3)}{\sqrt{3}C_0}$ .  
For convenience, let us introduce the following notations:
$$
{\bf{\Gamma}}:=\bigsqcup_j \gamma_j;\
{\bf T}:=\bigsqcup_j T^2_j;\
{\bf R}:=\bigsqcup_j {R}_j.
$$

By  Corollary \ref{number}  we obtain  the following upper bound on the length of ${\mathbf\Gamma}$:
\begin{equation}\label{length}
	l({\bf{\Gamma}})\leq C_{{\bf{\Gamma}}}:=C\cdot \frac{4H_0Vol(M^3)} {\sqrt{3}C^2_0}=\frac{8H^2_0{Vol(M^3)^2}} {{3}C^3_0}. 
\end{equation}

 Let us denote 
	\begin{equation}\label{1/2}
		{\bf B}:= \{\theta \in S^1| card(u|_{{\bf{\Gamma}}})^{-1}(\theta) \leq 2 C_{{\bf{\Gamma}}}|| \alpha||_{L^{\infty}}\}.
	\end{equation}
	where $\alpha=u^*d\theta.$
	\begin{lemma}\label{L2}
	$
	\mu ({\bf B})>\frac{1}{2}.
	$
	
\end{lemma}

\begin{proof}
	First note that $ ||\alpha||_{L^{\infty}}$ is equal to the norm  $||du||_{L^{\infty}}= \max_{p\in M^3}|du|_p.$
	Assume that the statement of Lemma \ref{L2} is not true.  Then we have
	\begin{equation}{\label{card}}
		\mu (\theta \in S^1 | card(u|_{{\bf{\Gamma}}})^{-1}(\theta) > 2 C_{{\bf{\Gamma}}}|| \alpha||_{L^{\infty}})\geq\frac{1}{2}.
	\end{equation}

	Since ${\bf{\Gamma}} $ is compact, it follows from  Remark \ref{r2} that the set of regular values $reg(u|_{{\bf{\Gamma}}})$  of the function $u|_{{\bf{\Gamma}}}$ is an open and everywhere dense set in $S^1$\footnote{A value is considered regular if its preimage is empty.}. Recall that nonempty open sets in $S^1$ are either all $S^1$ or a finite or countable disjoint union of open intervals in $S^1$:
	\begin{equation}
		reg(u|_{{\bf{\Gamma}}})=\bigsqcup_{ \omega\in \Omega}J_{\omega},
	\end{equation}
	where $\Omega$ is either a finite or a countable indexing set, and $J_{\omega}$  either is an open  interval in $S^1$  for each $\omega\in \Omega$ or is entire  circle $S^1$. Clearly, in the last case $\Omega =\{\omega \}$.
	
	Since the mapping $u|_{{\bf{\Gamma}}}:{\bf{\Gamma}} \to S^1$ is a covering map on each preimage $(u|_{{\bf{\Gamma}}})^{-1}(J_{\omega})$, then by assumption \eqref{card} there is a subset $\Omega'\subset \Omega$ such that   the cardinality of the covering $(u|_{{\bf{\Gamma}}})^{-1}(J_{\omega})\to J_{\omega}, \ \omega \in \Omega', $ is greater than 	$2 C_{{\bf{\Gamma}}}|| du||_{L^{\infty}}$ and 
	\begin{equation}\label{12}
		\mu(\bigsqcup_{\omega \in \Omega'}J_{\omega})\geq\frac{1}{2}. 
	\end{equation}
	The  additivity of $\mu$ implies:
	
	\begin{equation}\label{l}
		\begin{array}{c}
			l({\bf{\Gamma}}) = \mu({\bf{\Gamma}})\geq \mu(\sum_{\omega\in\Omega'}(u|_{{\bf{\Gamma}}})^{-1}(J_{\omega}))=\sum_{\ \omega\in\Omega'}\mu((u|_{{\bf{\Gamma}}})^{-1}(J_{\omega}) \overset{\eqref{card}}>\\
			\\
			2{C_{{\bf{\Gamma}}}}|| du||_{L^{\infty}} \sum_{\omega\in\Omega'} \frac {1}{|| du||_{L^{\infty}}}\mu(J_{\omega}) = 2{C_{{\bf{\Gamma}}}}  \sum_{\omega\in\Omega'} \mu(J_{\omega})\overset{\eqref{12}}\geq C_ {{\bf{\Gamma}}}.
		\end{array}
	\end{equation}
	which contradicts to \eqref{length} and  proves  Lemma \ref{L2}.

\end{proof}

From the lemmas \ref{L1} and \ref{L2} we immediately obtain the following corollary.

\begin{corollary}\label{lll}
 We can find the value  $\theta_0 \in {\bf A} \cap {\bf B}$  such that $\theta_0$ is regular value for $u, u|_{\bf T}, u|_{\bf {\bf\Gamma}}$.  
\end{corollary}
\begin{proof}
	Since $\mu(S^1) =1$, by  the measure property we have $$\mu({\bf A}\cup {\bf B})=\mu( {\bf A}) + \mu( {\bf B}) - \mu( {\bf A} \cap {\bf B}) \leq 1 ,$$ which implies    $\mu( {\bf A} \cap {\bf B})>0$. 	 The rest follows from Remark \ref{r2}. 
\end{proof}

Let us emphasize the  following properties of $\Sigma_{\theta_0}$: 

\begin{itemize}
	\item  $-\chi(\Sigma_{\theta_0})\leq \frac{1}{2\pi}||\alpha||_{L^2}||R^-||_{L^2}$.
	\item If  $x\in \Sigma_{\theta_0}   \cap {\bf{\Gamma}} $,  then   $   {\bf{\Gamma}} \pitchfork \Sigma_{\theta_0} $ at the point $x$. 
	\item If $\Sigma_{\theta_0}\cap {\bf T }\not =\emptyset$, then  $  \Sigma_{\theta_0} \pitchfork  {\bf T }$.
\end{itemize}

\begin{definition}\label{circ}
	Denote by   ${\mathcal C}=\{C_j\}$  the disjoint finite family  (possibly empty) of circles  such that  $\Sigma_{\theta_0}\cap {\bf T } =\bigsqcup_j C_j$. 	
\end{definition}

\begin{corollary} \label{fincirc}
	The number of those circles of the family ${\mathcal C}$, that represent the nontrivial kernel $\ker( {\bf i}_*: H_1({\bf T };\Z)\to H_1(\bf R;\Z))$ does not exceed $2{C_{{\bf{\Gamma}}}}||\alpha||_{L^{\infty}}$, where ${\bf i}_*$ is a homomorphism  induced by the embedding  ${\bf i}: {\bf T} \hookrightarrow {\bf R}$.
\end{corollary}

\begin{proof}
	The proof follows  immediately  from  the definition of the set $\bf B$ (see \eqref{1/2}) and the fact that   ${\bf{\Gamma}}$ necessarily intersects each of the circles in the  family $\mathcal C $, which represent the non-trivial kernel $\ker ( {\bf i}_*: H_1({\bf T };\Z)\to H_1({\bf R} ;\Z))$.
	
\end{proof}

\begin{proposition}\label{prop2}
	Let  $i:M^2 \hookrightarrow M^3$ be an  embedding such that $i(M^2)=\Sigma_{\theta_0}=u^{-1}(\theta_0), $ where $\ \theta_0 \in S^1$ from  Corollary \ref{lll}.    Then there is an embedding of general position $i': M^2 \hookrightarrow  M^3$ with the image  $\Sigma'_{\theta_0}:= i'(M^2)$ satisfying the following properties:
	\begin{enumerate}
		\item [1)]$\Sigma'_{\theta_0}\simeq M^2$. In particular,  $-\chi(\Sigma'_{\theta_0})\leq \frac{1}{2\pi}||\alpha||_{L^2}||R^-||_{L^2}$;
		\item [2)] If $\Sigma'_{\theta_0}\cap {\bf T }\not =\emptyset $ then $\Sigma'_{\theta_0}\pitchfork {\bf T }$  and  the intersection   $\Sigma'_{\theta_0}\cap {\bf T}$ is a disjoint  union of circles ${\mathcal C'}=\bigsqcup_jC'_j$;
		\item [3)]The number of those circles of the family ${\mathcal C'}$  that   represent the nontrivial kernel $\ker( {\bf i}_*: H_1({\bf T };\Z)\to H_1(\bf R ;\Z))$ does not exceed $2{C_{{\bf{\Gamma}}}}||\alpha||_{L^{\infty}}$, where  $\alpha = u^*d\theta$;
		\item [4)]$[M^2,i'] =[M^2,i]\in H_2(M^3;\Z)$.
		
	\end{enumerate}
	
\end{proposition}
\begin{proof} 		For simplicity, we identify $M^2$ with $i(M^2)$.  Let us consider a tubular neighborhood $W\subset M^3$ of the submanifold $M^2$ such that $W\cap {\bf T}$ consist of disjoint tubular neighborhoods $\{W_j\simeq C_j\times \R\}$ in $\bf T$  of the finite family   of circles ${\mathcal C}=\{C_j\}$   defined in  Definition \ref{circ}.  Since $M^2$ and $M$ are orientable, $W$ is diffeomorphic to the trivial normal bundle $\nu M^2$  over  $M^2$. We can identify $W$ with the direct product $ M^2\times \R$, where  $M^2$ corresponds to the zero section $M^2\simeq M ^2\times 0 \overset{i_W}\hookrightarrow  M^2\times \R \simeq W$.  Let us  identify the pair $(W,\bigsqcup_jW_j)$ with the pair of linear bundles  $(\nu M^2, \nu M^2|_{\sqcup_jC_j})$.  
	
	Let $p:W\to M^2$ be a projection along the fibers of $W$.  Recall that the identity component $Diff_0 ^2 (M^2, M^2)$ of $C^2$ - diffeomorphisms  $Diff ^2 (M^2, M^2)$ is  open  in $C^2(M^2, M^2)$ (see \cite {Hi}) and  it's preimage under the continuous mapping $C^2(M^2, W)\overset {p_*}\to C^2 (M^2, M^2)$,  which is defined by  $p_*(f) = p\circ f$,  is an open neighborhood $V_1$ of the  zero section $i_W:M^2 \to W$  (see \cite{PM}). Clearly, $V_1$ consists  of some family of  embeddings $M^2 \to W$ transversal to the fibers of $W$.   
	
	Since ${\bf T}\cap W$ is a closed subset of $W$,  the subset of $C^2(M^2,W)$ transversal to ${\bf T}\cap W$ is  open in  $C^2(M^2,W)$ - topology (see \cite {PM}).  Denote it by $V_2$.   
	Let $i'_W:M^2 \to W$  satisfy the condition $I$ and $II$ of  Theorem \ref{CC1} and $i_W'\in V_1\cap V_2$.  	   	    
	Let us put $i':=i^W\circ i_W'$, where $i^W: W\hookrightarrow M^3$ is a natural embedding.   Denote by $\Sigma'_{\theta_0}$ the image $i'(M^2)\subset M^3$.  	From the properties of $V_1$ and $V_2$ it follows  that   each fiber of $W$  transversely intersects   the embedded submanifold $\Sigma'_{\theta_0}$ exactly at one point and the  parts $1$ and $2$  immediately follow. 
	Since the fibers   of the  bundle  $W_j$ are the fibers   of $W$,  then  $\Sigma'_{\theta_0} \pitchfork W_j$ and  $\Sigma'_{\theta_0}\cap W_j$ is a circle $C'_j$  transversal to the fibers of  $W_j$ for each $j$ and therefore $C'_j$ is homotopic to $C_j$ in $W_j$.   If the circles $C_j$ and $C'_j$ are equipped by   the corresponding orientation, then   $[C_j] = [C_j']\in H_1(\bf T;\Z)$.   Now the statement of part $3$ immediately   follows from Corollary \ref{fincirc}.
	Since an arbitrary diffeomorphism belonging to $Diff ^2_0(M^2,M^2)$ induces the identity isomorphism of $H_2(M^2;\Z)$ and the embeddings $i$ and $i'$, up to such a diffeomorphism  differ in deformation along the fibers $W$, part $4$ is proved.

\end{proof}

\subsection { Surgeries }\label {3.3}

Let $i': M^2\hookrightarrow M^3$  be  a  general position  embedding    from Proposition \ref{prop2} and  $l_1\in M^2$ be an inessential closed  orbit of $\F'=i'^{-1}(\F\cap i'(M^2))$ such that $0\not =[l_1]\in \pi_1(M^2,y_1), \ y_1\in l_1$.  Since $l_1$ is inessential, due to Jordan-Sch\"onflies theorem,   $i'(l_1)$  bounds a disk  in it's support  $ L\in\F$.  Moreover, due to  Reeb's stability theorem
there is a {\it good neighborhood}  $V_{l_1}\simeq l_1\times (-\e,\e)$ in $ M^2$, i.e., a  neighborhood   fibered by the  inessential  closed orbits $ l_1\times t, \ t\in (-\e,\e)$. 

Let us choose a nonzero value ${\e_1}\in (0,\e)$ and produce a surgery on $M^2$ cutting 
out $V_{1} \simeq  l_1\times (-{\e_1},{\e_1})\subset l\times (-\e,\e)\simeq V_{l_1}$ and gluing in the disks ${\mathcal D}_{1}\bigsqcup {\mathcal D}_{-1}$.   Denote by $M^2_1$ the obtained manifold. 
Then  we find  next  inessential closed  orbit $l_2\subset M^2_1$ (if such  exists)  with the good collar $V_{l_2}\simeq l_2\times (-\e,\e)$ such that $0\not =[l_2]\in \pi_1(M_1^2,y_2), \ y_2\in l_2$.    Choosing a nonzero value $\e_2 \in (0,\e)$ we make a surgery cutting 
out $V_{2} \simeq  l_2\times (-{\e_2},{\e_2})\subset l_2\times (-\e,\e)\simeq V_{l_2}$ and gluing  the disks ${\mathcal D}_{2}\bigsqcup {\mathcal D}_{-2}$ instead.  We obtain the new manifold $M^2_2$.  Then we select the next curve $l_3\subset M^2_2$ with the same properties  and follow the same steps as above up to getting a manifold $M^2_{\rho}$.

Let $\{{\mathcal D}_{\pm i}\}, \ i\in \{1,\dots,\rho\},$   be  a  family of the disjoint disks  surgically pasted instead of the cut out annuli  $V_i\simeq l_i\times (-\e_i, \e_i)\subset l_i\times (-\e,\e)$, where $l_i\subset M^2_{i-1}$ is an inessential closed orbit such that $0\not =[l_i]\in \pi_1(M_{i-1}^2,y_i), \ y_i\in l_i$.  Denote $l_{\pm i} = \partial {\mathcal D}_{\pm i}$. Let us  endow $M^2_{\rho}$ with the structure of an  differentiable oriented manifold, complementing the differentiable structures and corresponding orientations of  disks $\bigsqcup^{\rho}_{i=1}{\mathcal D}_{\pm i}$ and $M^2\setminus \bigsqcup^{\rho}_{i=1}V_i $ with a differentiable structure and  an agreed orientation of a tubular neighborhood of the boundary $\partial (M^2\setminus \bigsqcup^{\rho}_{i=1}V_i) $ (see \cite{Hi}).

Let us extend  $i'|_{M_{\rho}^2\setminus \Int {\bigsqcup^{\rho}_{i=1}{\mathcal D}}_{\pm i}}=i'|_{M^2\setminus \bigsqcup^{\rho}_{i=1} V_i}$ to all of $M_{\rho}^2$ by  embeddings  $h_{\pm i}: {\mathcal D}_{\pm i}\to M^3$, such that $h_{\pm i}({ {\mathcal D}}_{\pm i}) = D_{\pm i}$, where $D_{\pm i}\subset L_{\pm i}\in \F,$ are  disks in the corresponding leaves of $\F$ such that $i'(l_{\pm i}) =\partial D_{\pm i},  \ i\in \{1,\dots,\rho\}$.  

Let us consider arbitrary small   foliated neighborhoods $U_{\pm i}$ of $D_{\pm i}$.   Applying  an isotopy to $h_{\pm i}$  that is  supported in $\mathcal D_{\pm i}$ and has  a values in  $U_{\pm i}$,   which pushes out $D_{\pm i}$ into the side that $i'(V_i)$ belongs  to,  we can obtain a smooth general position  immersion $i'_{\rho}: M_{\rho}^2 \to M^3$ which is a  continuation of    $i'|_{M_{\rho}^2\setminus \Int {\bigsqcup^{\rho}_{i=1}  {\mathcal D}}_{\pm i}}$  such   that      the induced foliation ${i'_{\rho}}^{-1}(\F\cap i'_{\rho}({\mathcal D}_{\pm i}))$  on each ${ \mathcal D}_{{\pm i}}$    consists  of  inessential closed orbits   surrounding a center   and  the immersion  $i'_{\rho}$  is  still  transversal to $\bf T$.


\begin{lemma}\label{LL}
	$[M^2_{\rho},i'_{\rho}] =[M^2,i'] \in H_2(M^3;\Z)$.
\end{lemma}
\begin{proof}
	The singular cycles  $(M^2,i')$ and $(M_{\rho}^2,i'_{\rho})$ differ by  the sum of spherical cycles   $\oplus^{\rho}_{i=1}  (S_i^2, g_i)$  where   $S_i^2$  is identified with an annulus $A_i \cong \bar V_i$  to which  two disks $ \mathcal D_{\pm i}$ are glued by identifying  the  boundaries.  Put   $g_i|_{A_i} = i'$ and $g_i|_{ \mathcal D_{\pm i}} =i'_{\rho}$.  From irreducibility  of $M^3$ it follows that $g_i$ can be extended to a mapping  of the ball $G_i: D_i^3\to M^3$.   Taking into account the orientation  coming from  $M^2$ and $M_{\rho}^2$, on the level of singular chains we have  $\partial (\oplus^{\rho}_{i=1}(D_i^3, G_i)) = \oplus^{\rho}_{i=1}(S_i^2, g_i)$, which  implies  the result. 
\end{proof}

\begin{definition}\label{N2}
	Denote by $(N^2,p)$ the  the singular cycle  $(\hat M^2_{\rho},i'_{\rho})$, where $\rho$ is maximal number of surgeries described above and $\hat M^2_{\rho}\subset M^2_{\rho}$ consists of connected components of $M^2_{\rho}$ representing nontrivial homology.    
	As usual, let $\F'$  denote the induced foliation  $p^{-1}(\F\cap p(N^2))$.  
\end{definition}

\begin{remark}\label{constr}
	By Lemma \ref{LL} we have 	$[N^2 ,p] =[M^2,i'] \in H_2(M^3;\Z)$. By the construction,  taking into account  the Jordan-Sch\"onflies theorem and irreducibility of $M^3$, each inessential closed orbit of $\F'$ must bound a disk in $N^2 $ and $N^2$ does not contain spherical components.  
\end{remark}

Let everywhere below $N^2,\F'$  and $p$ satisfy Definition \ref{N2}.

\subsection{Maximal vanishing cycles}

Let $\mathcal O =\partial \p \subset N^2$ be a vanishing cycle (see Definition  \ref{VC}).  	Note that $\p$ is  uniquely defined by $\mathcal O$  because  an ambiguity can only arise when $\mathcal O$ is a closed orbit of $\F'$ and the connected component of $N^2$ containing  $\mathcal O$ is a sphere which is impossible.  In this case, we will understand by $\p(\mathcal O)$  the set  $\p$  from Definition \ref{VC}  bounded by the vanishing cycle $\mathcal O$. 

Let us  introduce the notion of {\it maximal vanishing cycle}.

\begin{definition}\label{omax}
	A vanishing cycle   ${\mathcal O}_{\max}\subset N^2$ is called {\it maximal}  	 if  
	$$\p({\mathcal O}_{\max})\subset \p(\mathcal O) \  \text{implies} \ {\mathcal O}_{\max} = \mathcal O.$$ 
\end{definition}

From definition \ref {omax} immediately follows: 

\begin{lemma}
	  ${\mathcal O}_{\max}$ is either essential closed orbit of $\F'$ whose $p$-image is a nontrivial vanishing cycle,  or it is a singular, i.e. consisting  of separatrix loops.
	\end{lemma}
\begin{proof}
	  Indeed, otherwise  due to Reeb's stability theorem  ${\mathcal O}_{\max}$ is an inessential closed orbit having a good collar consisting of inessential closed orbits containing a vanishing cycle  $\mathcal O=\partial \p(\mathcal O)$ different from ${\mathcal O}_{\max}$ such that    $\p({\mathcal O}_{\max})\subset \p(\mathcal O)$, which is impossible. 
	\end{proof}

\begin{remark}\label{Max}
	
	If $\mathcal O_{\max}$  is  an essential, then
	by  Theorem  \ref{Nov},  $p({\mathcal O})\in {T^2}$, where $T^2$ is the boundary torus of a Reeb component $ R$ and $p_*[{\mathcal O}_{\max}] \in \ker( i_*: \pi_1(T^2)\to \pi_1({R}))$\footnote{ By the class $[{\mathcal O}] $ we mean the  class of the loop $f:S^1\to N^2$   which bypass ${\mathcal O} $ once along the trajectories of the vector field  tangent to $\F'$. }.  Since  the immersion  $p$ is  transverse to $\bf T$ by the construction,  then $\mathcal O_{\max}$ must be a regular vanishing cycle, i.e. a  closed orbit of $\F'$.  Therefore, in the case  where   $\mathcal O_{\max}$ is a singular,  it must be  inessential.   In particular, if $\mathcal O_{\max}$  consists of two separatrix loops  ${\mathcal O}_1$ and $ {\mathcal O}_2$, i.e. $\p({\mathcal O}_{\max})$ is a pinched annulus,  then  ${\mathcal O}_{\max}$ can be of two types:
	\begin{enumerate} \label{inessent}
		\item [A)]   Both ${\mathcal O}_1$ and ${\mathcal O}_2 $ are inessential.  
		\item [B)]  Both  ${\mathcal O}_1$ and ${\mathcal O}_2$ are essential and $p_*[{\mathcal O}_1] = - p_*[{\mathcal O}_2] \in \pi_1 (\mathcal L)$, where $\mathcal L\in \F$ is a a support of  $p({\mathcal O}_{\max})$.   Using  Jordan-Sch\"onflies theorem one can see  that   $p({\mathcal O}_{\max})$ must bound   a pinched annulus in $\mathcal L $.
	\end{enumerate}
	
\end{remark}
\
\

\begin{lemma}\label{ll1}
	Let  ${B} \subset N^2$  be  a disk of $N^2$  bounded by an  inessential closed orbit  of $\F'$.  Then
	$B \subset  \p(\mathcal O_{\max})$ for some maximal  vanishing cycle $\mathcal O_{\max}$.

\end{lemma}
\begin{proof}   Due to Reeb’s stability theorem  each inessential closed  orbit $l_0$ of $\F'$ has  a { good neighborhood}  homeomorphic to $(-\e,\e) \times l_0$, where $l_s = s\times l_0$ is an inessential closed orbit of $\F'$. Let $U=\bigcup _{ t }B_t, \ t\in \mathcal T$, be  the  union of  disks containing $B$, obtained by adding to $B$   annuli consisting of the  union of inessential closed orbits. Let  $l_{t } = \partial  B_t$.  Clearly,  the family of disks $\{B_t\}$ is linearly ordered by the  inclusion:    $t_1< t_2\Leftrightarrow B_{t_1}\subset B_{t_2}$.
	
	Note  that $\partial \bar U$ cannot be a center since $ N^2$ does not contain a connected component homeomorphic to $S^2$.   Observe also that 
	$\partial \bar U$ consists of orbits of $\F'$ which are not   inessential closed orbits  because  such a  closed orbits     have   good neighborhoods  and cannot belong to $\partial \bar U$. 
	
	Note that $\partial \bar U$ is a saturated set, i.e. it  consists of leaves of $\F'$ (see \cite{Sib}). If the closure $\partial \bar U$ contains     a   regular  leaf $r \in \F'$ to which  accumulate  another leaves of $\partial \bar U$,  then there exists a  small transversal $\tau$ to $\F'$ through $r$ which contains   the  interval $J$  connecting  two points  $ a\in l_{t_1}, \ b\in l_{t_2},$   between which  there are points of $\partial \bar U$.  But  it is  impossible because  $\F'^{\perp}$ is not degenerated on  $l_t$ and $l_t$ separates $N^2$, therefore if $\tau $ leaves $B_t$, it will never return to $B_t$.   
	
	We conclude that $\partial \bar U$ consists of at most a finite union ${\bf O}: =\bigsqcup_i \mathcal O_i $ of  essential closed orbits or separatrix loops of $\F'$.   The claim to show  that ${\bf O}$ is  connected and is a vanishing cycle. Denote by  ${\bf U}:=\bigsqcup_i U_i$    a disjoint union  of  tube neighborhoods   $U_i$ of ${ O_i}$. Clearly, $\bar U \setminus \bf U$ is  compact and is contained inside of $B_{t_0}$ for some $t_0\in \mathcal T$. Since $U\setminus B_{t_0}$ is connected  we immediately conclude that ${\bf O}$ is connected.  From the orientibility of $N^2$ it follows that  ${\bf O}$ divides $\bf U$ into  connected components,   closure of each of which   in $N^2$  has  a nonempty boundary consisting  of orbits  of ${\bf O}$.    Since ${ U}\setminus B_{t_0}$ is connected it can only belong  to one of these connected components  and thus, by the definition, ${\bf O}$  is a vanishing cycle. The result follows from the finiteness  of both  the set of  singular points and the  number of essential regular vanishing cycles of $\F'$.
\end{proof}

\begin{lemma}\label{ll2}

	Let  $\p_{\max}=\p({\mathcal O_{\max}})$ and $\p'_{\max}=\p({\mathcal O'_{\max}})$, where $O_{\max}$ and $O'_{\max}$ are  maximal vanishing cycles. Then 	either $\p_{\max}=\p'_{\max}$ or $\p_{\max}\cap \p'_{\max}=\emptyset$. In particular,  $\p_{\max}$ in  Lemma \ref{ll1} is unique. 
	
\end{lemma}
\begin{proof}

	It is enough to  suppose that  ${\mathcal O}_{\max}$ and $ {\mathcal O'}_{\max}$ are different.  Otherwise we obtain a  contradiction since  $N^2$  does not contain $S^2$ as a connected component. One of the following  takes place for  ${\mathcal O}_{\max} \cap {\mathcal O'}_{\max}$:  
	
	\begin{enumerate}
		\item [(i)]$\emptyset$;  
		\item [(ii)] a saddle point;
		\item [(iii)] a separatrix loop.  
	\end{enumerate}
	
	In the cases  $(ii)$ or $(iii)$  at least one of    $\p_{\max}$ or  $\p'_{\max}$   must be  a  disk and  ${\mathcal O}_{\max}\cup {\mathcal O}'_{\max}$ is  two separatrix loops with  a common saddle point $s$. By Remark \ref{inessent},   $\mathcal O$ and $\mathcal O'$  are inessential and therefore due to the Reeb stability theorem there exists an external good collar $V$  of $\p_{\max}\cup \p'_{\max}$ \footnote{Note that $\p_{\max}\cup \p'_{\max}$ is homeomorphic to either a disk or  a bouquet of two disks.}. Let $l\subset V $ be an inessential closed orbit. 
	Clearly, $l$ bounds a disk $B$ containing  $\p_{\max}\cup \p'_{\max}$. Applying Lemma \ref{ll1} we find a vanishing cycle ${\mathcal O}$ such that  $\p_{\max}\cup \p'_{\max}\subset \p({\mathcal O})$ which contradicts  the maximality of both ${\mathcal O}_{\max}$ and ${\mathcal O}'_{\max}$.

	Let us consider   case $(i)$.  
	Let us suppose that there exists $a\in \Int \p \cap \Int \p'$. Since  $\p$ and $\p'$ are connected,   $\p'_{\max}\not \subset \p_{\max}$ and $\p_{\max}\not \subset \p'_{\max}$,  we have  ${\mathcal O}_{\max}\cap \p'_{\max}\not=\emptyset$ and     ${\mathcal O'}_{\max}\cap \p_{\max}\not = \emptyset$.   Taking into account the condition $(i)$ and connectivity of ${\mathcal O}_{\max}$ and  ${\mathcal O'}_{\max}$ we obtain $${\mathcal O}_{\max}\subset \Int \p'_{\max} \  \text{ and }  \  {\mathcal O'}_{\max}\subset \Int \p_{\max}.$$  Let  $l\subset \p_{\max}$ be an inessential closed orbit of a good collar of $\mathcal O_{\max}$, which bounds a disk $B$ inside of $\p_{\max}$ such that   $a\cup {\mathcal O'}_{\max}\subset \Int B$.  Since $\p'_{\max}\not \subset B$,  for reasons similar to the above, we conclude that $l\subset \Int \p'_{\max}$.  By the Jordan-Sch\"onflies theorem $l$ bounds a disk  $B'\subset \Int\p'_{\max}$. On the  other hand,  $l$ bounds $B\subset \p_{\max}$. Since  $ {\mathcal O'}_{\max}\subset \Int B$ we conclude that $B\not = B'$ which  implies that  $B \cup B' \simeq S^2$. But this contradicts to the fact that $N^2$ does not contain  connected components homeomorphic to the sphere.   Thus it follows that   $ \Int \p \cap \Int \p'=\emptyset$ which implies  the result.
	
\end{proof}

\begin{corollary}\label {cmax}
	Each center of $\F'$  belongs to the unique  $\p_{\max}=\p({\mathcal O}_{\max})$.  
\end{corollary}

\begin{proof} A center  $\F'$ has a punctured neighborhood consisting of inessential closed orbits and the result immediately follows from  Lemmas \ref{ll1} and  \ref{ll2}. 
\end{proof}

\begin{lemma}\label {ll4}

	Let $\p_{\max} = \p({\mathcal O_{\max}})\subset N^2$ be a pinched annulus.   Then the   separatrix loops of $\mathcal O_{\max}$  are essential and  their $p$ - images   bound a pinched annulus in the leaf $\mathcal L\in \F$ containing  $p(\mathcal O_{\max})$.

\end{lemma}

\begin{proof}  
	According to  Remark \ref{Max}  it is enough to show that there is no  maximal vanishing cycle ${\mathcal O_{\max}}$  consisting of inessential separatrix loops.  
	
	Suppose that  the  separatrix loops of $\mathcal O_{\max}$ are inessential,  then, due to  Reeb's stability theorem, they have  good exterior  collars with respect to the  pinched annulus $\p_{\max}$. By Remark \ref{constr} each  closed orbit  of this collar  must bound  a disk in $N^2$.   Since there are no connected components of $N^2$  homeomorphic to $ S^2$, one of such disks contains $\mathcal O_{\max}$.   We conclude that   ${{\mathcal O}_{\max} }\subset \Int \p(\mathcal  O)$ for some vanishing cycle $\mathcal O$, which  contradicts the maximality of    ${{\mathcal O}_{\max} }$.

\end{proof}

\section {Proof of  Main Theorem}

\subsection{Reducing the number of singular points. }
\label{red}
Let $\{\p^k_{\max}=\p({\mathcal O}^k_{\max}), \ k\in \bf K\}$ be a family of  disks and pinched annuli in $N^2$ bounded by maximal vanishing cycles of $\F'$, where $\bf K$ denotes a finite (possibly empty) indexing set.  Let $\{V_k \subset \p^k_{\max}, \ k\in \bf K\}$ denote good collars of ${\mathcal O}^k_{\max}$ and $\{l_k\subset V_k, k\in \bf K\}$   are fixed inessential closed orbits of $\F'$ inside of the good collars.  Let us suppose that $V_k$ is small enough to $p|_{V_k}$ be an  embedding.  By    Remark \ref{constr}, Definition \ref{essen} and Jordan-Sch\"onflies theorem, each $l_k$ bounds a disk $B_k$ in $N^2$, and  $p (l_k)$  bounds a disk $D_k\subset L_k\in \F$ in the supporting leaf $L_k\in\F$.   Let us  redefine the mapping $p|_{ B_k}$ by the embedding $h_k:  B_k \to M^3$ such that $h_k|_{l_k}= p|_{l_k}$ and  $h_k(B_k) =D_k$.

Let us consider arbitrary small   foliated neighborhoods $U_{k}$ of $D_{k}$.   Applying  an isotopy to $h_{k}$  that is  supported in $B_{k}$ and has  a values in  $U_{k}$,   which pushes out $D_{k}$ to the side away from $p(V_k)\cap U_{k}$,    we can obtain a smooth general position  immersion $p': N^2 \to M^3$  which is a  continuation of    $p|_{N^2\setminus \Int {\bigsqcup_k  {B}}_{k}}$  such   that      the induced foliation ${p'}^{-1}(\F\cap p'({B}_{k}))$  on each ${B}_{{k}}$   consists  of  inessential closed orbits  surrounding a center $c_k$.

\begin{lemma}
	$[N^2,p] = [N^2,p'] \in H_2(M^3;\Z).$
\end{lemma}

\begin{proof}
	For each $k\in \bf K$  let $S^2_k:= ( B^1_k \bigsqcup  B^2_k)/ (\partial  B^1_k \sim \partial  B^2_k) \simeq S^2$ be  two copies of $ B_k$ with naturally identified boundaries.   
	Let us define a spheroid $g_k: S^2_k\to M^3$, where $g_k|_{ B^1_k}=p|_{ B^1_k}$ and $g_k|_{ B^2_k}=p'_{ B^2_k}$.    Since $M$ is irreducible, $g_k$ can be extended to a mapping of the ball: $\Phi_k :D^3_k\to M^3 $  such that  $S^2_k=\partial D^3_k$.  Taking into account the orientations of $ B^i_k,\ i=1,2,$  coming from  the  orientation of $ B_k$, on the level of singular chains we obtain $\partial (D^3_k,\Phi_k) =(S_k^2, g_k)$. It means that   $(N^2,p) -( N^2,p')=\partial ( \oplus_k (D^3_k,\Phi_k))$   which implies the result.

\end{proof}

\begin{definition}
Let us denote $\F'':= p'^{-1}(\F\cap p'(N^2)).$  
\end{definition}
Let  ${\bf K}'\subset \bf K$ be such that $$\{\p^k_{\max}=\p({\mathcal O}^k_{\max}), \ k\in \bf K'\subset \bf K\}$$ is a family of disks or  pinched annuli, such that  each  ${\mathcal O}^k_{\max}$ is singular with a saddle $s_k$.    Let   $(\p_{\max},{\mathcal O}_{\max}, V, l, L, D, B, U, h,c,s)$  be an  arbitrary element of    
$\{(\p^k_{\max},{\mathcal O}^k_{\max}, V_k, l_k, L_k, D_k, B_k, U_k, h_k, c_k, s_k), \ k\in \bf K'\}.$  
From  Remark \ref{inessent} and Lemma \ref{ll4} it follows that $p'({\mathcal O} _{\max})$ also bounds respectively a disk or  a pinched annulus  in its support $L\in \F$, which we denote by $D_{\max}$.  

Let  us suppose that $D_{\max}$ is a pinched annulus, then  
$D_{\max}\subset A \subset L$ , where $A\simeq S^1\times (0,1)$ be an annular  neighborhood of $D_{\max}$ in the leaf $L$ and $D_{\max}$ is a deformation retract of $A$.  Since  the collar $V$ of ${\mathcal O}_{\max}$  can be taken arbitrarily small,  we can assume that the  normal  relative to   $\F$ collar  $N\simeq  A\times [0,1)$ of $A=A\times 0$ contains $p'(V)$ and the foliation $\F\cap {N}$ is transversal to the interval fibers $\{*\times [0,1)\}$.  The  embedding $$S^1:=S^1\times 1/2\hookrightarrow S^1\times (0,1)\simeq A$$   extends to the embedding  $S^1\times [0,1) \hookrightarrow  A\times [0,1)\simeq N$  transversal to $\F\cap {N}$.  The image of this embedding we also denote by $S^1\times [0,1) $.  Clearly,    the foliation  $\F\cap N$  is obtained from the foliation $ \F \cap (S^1\times [0,1))$  by multiplying  by the  interval $(0,1)$. Since  leaves of   $\F \cap (S^1\times [0,1))$ are homeomorphic to  intervals or circles  representing the generator of $\pi_1(S^1\times [0,1))\cong \Z$,  the foliation $\F\cap {N}$ consists of leaves, that are either    homeomorphic to    annuli, which are deformation retract of $N$ or  contractible. 
It follows that each leaf  $\cal L$ of $\F\cap {N}$ induces a monomorphism of fundamental groups with respect to the embedding ${\cal L}\longrightarrow N$.
Therefore, since the loop $p'(l)$ is free homotopic to the loop $p'({\mathcal O}_{\max})$ inside of $N$, and  the loop $p'({\mathcal O}_{\max})$ is null-homotopic in $A$,  the loop  $p'(l)$ is null-homotopic  in $N$ and therefore, it is null-homotopic   in its support    ${\cal L} \in \F\cap N$\footnote{$L\cap N$ can be disconnected.}. Thus,  by the Jordan-Sch\"onflies theorem $p'(l)$ bounds a disc in $\cal L$.  Since there is no leaves of $\F$ homeomorphic to sphere, this disc should be coincide with the disk $D$. 

In the case  ${ D}_{\max}$ is  a disk,  we denote by   $A$  an open disk in $L$ containing $D_{\max}$.  Then  due to Reeb's stability theorem the induced foliation $\F\cap N$ of the normal collar $N\simeq A\times [0,1)$ containing $p'(V)$ is  homeomorphic to the product foliation $\{A\times *, \ *\in[0,1)\}$, i.e. is a foliation by disks and    by the Jordan-Sch\"onflies theorem $p'(l)$  will also bound the disc $D$ in its support   ${\cal L}\in \F\cap N$.

Since $U$ can be taken by an arbitrarily small neighborhood of $D$ we can assume that $p'(B)\subset N$. 
Let us denote $B_{\max}:=p'(B\cup V)$.


	By the construction,  in the case ${ D}_{\max}$ is a pinched annulus,  ${D}_{\max}\cup { B}_{\max}$  bounds   a ball $Q^3$ with two points identified which we call a pinched ball.  Using the same reasoning as for the disk $D$ we can show that  the  foliation  $\F\cap Q^3= \{D_t, \ t\in [0,1]\}$, is a foliation by disks excepting the cases $t=0$, where $D_0 = D_{\max}$, and $t=1,$ where $D_1 = p'(c)$.  By the diffeomorphism we  can represent $(N,\F\cap N)$ in $\R^3$ in  such a way that the  foliation $\F\cap N$   becomes transverse to the vertical direction and $D_{\max}$ belongs to the horizontal plane (see Fig. \ref{ris1})\footnote {Recall, that $\F$ is transversally oriented.}. 
	
	If ${ D}_{\max}$ is homeomorphic to a disk, then ${D}_{\max}\cup { B}_{\max}$ bounds the ball $B^3$.  By the diffeomorphism we  can represent $(N,\F\cap N)$ in $\R^3$ in  such a way that the   foliation $\F \cap N$ becomes the level set of the height function and is a foliation into disks that degenerate to a point (see Fig. \ref{ris1}).

	\begin{figure}
		{\includegraphics[scale=0.5]{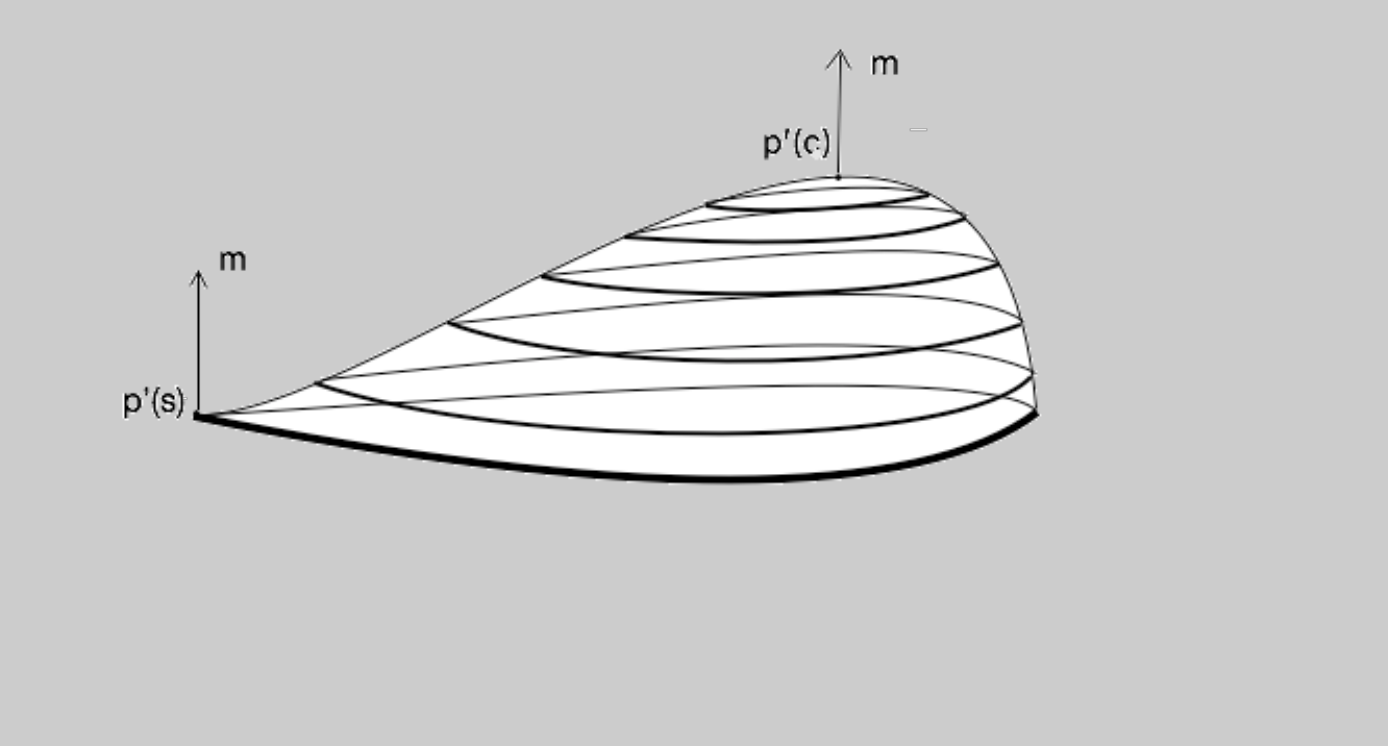}}
		{\includegraphics[scale=0.5]{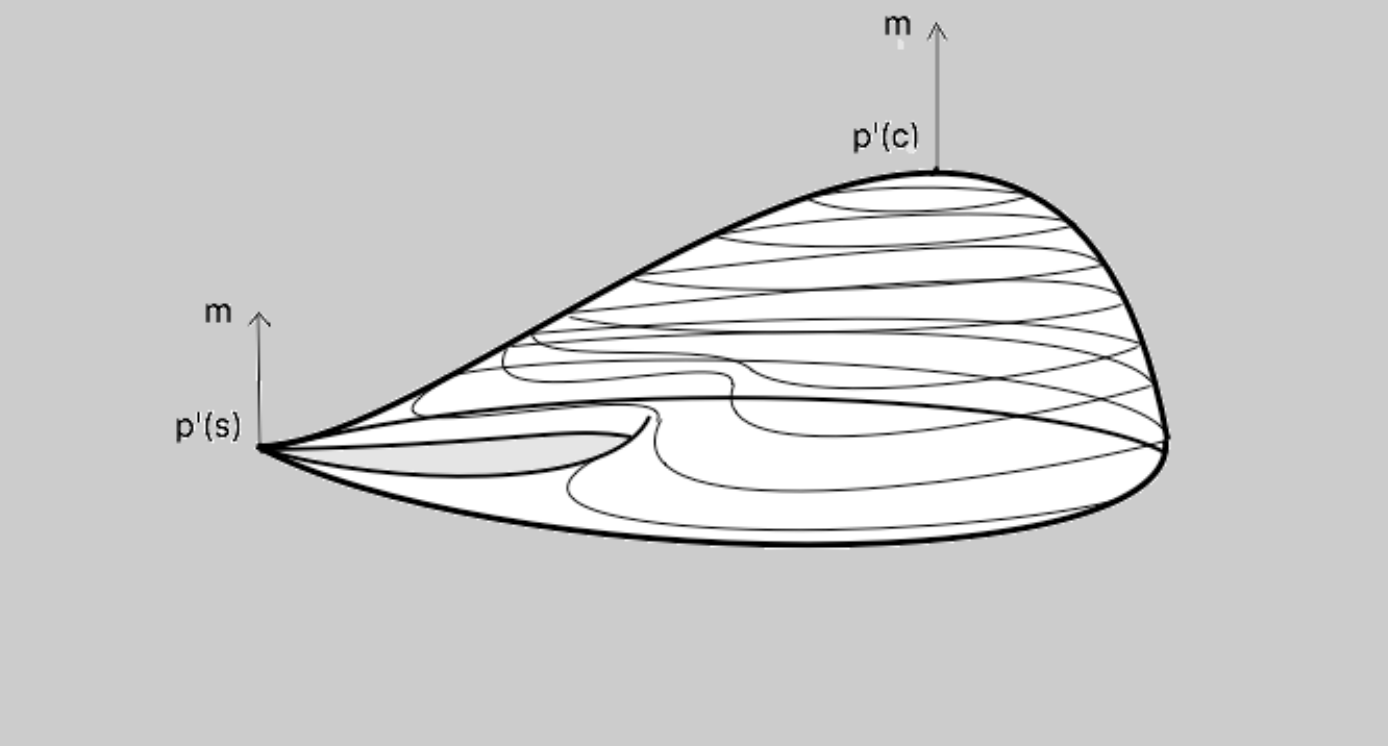}}
		\caption{The foliated ball $B^3$ and pinched ball $Q^3$}
		\label{ris1}
	\end{figure}
	
	Taking into account the form of a  surface in general position  with respect to the foliation  in the neighborhood of  singular points,  in  both cases, we can see that the  directions of the  normal vector field $n$ to the foliation $\F$ and the normal vector field $m$ to the $B_{\max}$ at the singular points $p'(s)$ and $p'(c)$ either  simultaneously coincide  or simultaneously opposite (see Fig. \ref{ris1}).  Thus the  types of  the singular points $s$ and $c$   coincide. 
	Since by Lemma \ref{ll2} the saddle point $s$ belongs to only one $\p_{\max}$, we conclude  that when calculating the Euler class the  pair of singular points $s$ and $c$ can be eliminated  because  their total index in the sum \eqref{index} is equal to  zero.

	\subsection {Estimation of  $L^2$-norm of the Euler class $e(\F)$.}

	Note, that   we  did not generate    new  (i.e.   not coming  from $(M^2, \F')$) essential closed orbits of  $(N^2,\F'')$.    Moreover, the  surgeries described in  subsection \ref{3.3} increase the Euler characteristic.  Taking into account Proposition \ref{prop2},   Remark \ref{Max} and  Corollary \ref{cmax}, we conclude that   the  number of centers of $\F''$ which are not eliminated above (see subsection \ref{red}), i.e.,   centers   corresponding to maximal regular vanishing cycles,   does not exceed $2 C_{{\bf{\Gamma}}}|| \alpha||_{L^{\infty}}$.   Since     
	$$-\chi(N^2)\leq -\chi(M^2)\leq \frac{1}{2\pi}||\alpha||_{L^2}||R^-||_{L^2},$$ 
	using the formulas  \eqref{index} and   \eqref {PH},  considering the singularities eliminated above, we get the following  estimate:
	\begin{equation}\label{euler}
		|e(T\F)([N^2,p'])|  \leq \frac{1}{2\pi}||\alpha||_{L^2}||R^-||_{L^2}  +  4C_{{\bf{\Gamma}}}|| \alpha||_{L^{\infty}} . 
	\end{equation}
		Taking into account \eqref{*-norm}, \eqref{Pet} and \eqref{V} we obtain 
	
	\begin{equation}\label{euler1}
		||e(T\F)||_{L^{2}}  \leq \frac{1}{2\pi}||R^-||_{L^2}  +  4C_{\bf{\Gamma}}\frac{\Lambda}{\sqrt{Vol(M^3)}} 
	\end{equation}
	
	Since $R^-\geq 6 k_0$,  together with \eqref{length} this implies 
	
	\begin{equation}\label{euler1}
		||e(T\F)||_{L^{2}}  \leq -\frac{3}{\pi}k_0\sqrt{V_0}  +   \frac{32 H^2_0V_0^{\frac{3}{2}}} {3C^3_0}{\Lambda},
	\end{equation}
 where the constant $C_0$ is defined in \eqref{C_0}. Thus, putting $$C_1:= -\frac{3}{\pi}k_0\sqrt{V_0}  +   \frac{32 H^2_0V_0^{\frac{3}{2}}} {3C^3_0}{\Lambda},$$ we obtain the statement of  Theorem \ref{main}.

	\vspace{2cm}
	
	The author thanks A. Borisenko, S. Maksimenko and V. Yaskin for useful comments.


	

\begin{thebibliography}{99}


\bibitem{BK}
{V.Bangert,~ M. Katz},
{\em 	An Optimal Loewner-type Systolic Inequality and Harmonic One-forms of Constant Norm}.
{Communications in 	analysis and geometry } \textbf {12(3)} (2004), 703-732. 


\bibitem{B1}
{D. V. Bolotov},
{\em Foliations on closed three-dimensional Riemannian manifolds with a small modulus of mean curvature of the leaves}.  { Izv. RAN. Ser. Mat.}\textbf {86} (2022), 85-102.

\bibitem{B2}
{D. V. Bolotov},
{\em On foliations of bounded mean curvature on closed three-dimensional Riemannian manifolds}.  {Proceedings of the International Geometry Center,}\textbf {16(2)} (2023), 173-182.





\bibitem{BZ}
{Yu. D. Burago and V. A. Zalgaller},   {\em Introduction to Riemannian Geometry}. [in Russian], \ Nauka, St. Petersburg, (1994).


\bibitem{CC1}
{A. Candel and L. Conlon},
\emph{Foliations I}.  Graduate Studies in Mathematics,  V. \textbf {23}, American Mathematical Society,  Providence, (2000).

\bibitem{CC2}
{A. Candel and L. Conlon},
{\em Foliations II}.   Graduate Studies in Mathematics, V. \textbf {60}, American Mathematical Society,  Providence, (2003).




\bibitem{DC}
{do Carmo},
{\em Riemannian Geometry}.
{Birkh\:auser},  {Mathematics},  (1992).



\bibitem{Croke}
{Christopher B. Croke},
{\em Some isoperimetric inequalities and eigenvalue estimates}. 
{ Annales scientifiques de l'École Normale Supérieure}, Série, \textbf {13(4)} (1980), 419-435.


\bibitem{ES}
{J.Eells,~ J.H.Sampson},
{\em Harmonic mappings of Riemannian manifolds}.
{Am. J. Math. } \textbf {86} (1964), 621-657. 


\bibitem{ET}
{Y. Eliashberg and W. Thurston},
{\em Confoliations},   University  Lecture Series \textbf {13} Amer. Math. Soc., Providence, (1998).


\bibitem{Good}
{D. Goodman},
{\em Closed leaves in foliations of codimension one.},  
{Comm. Math. Helv.},  \textbf  {50}  (1975), 383-388.

\bibitem{H}
{A. Hatcher},
{Notes on Basic 3-Manifold Topology}, https://pi.math.cornell.edu/~hatcher/3M/3M.pdf, (2000).

\bibitem{HH}
{Hilbert Hector,~ Urich Hirsch},   
{\em Introduction to the Geometry of Foliations. Part B. Foliations of codimension one.}
Second edition,   {Aspects of Mathematics},  \textbf {97}, {Friedr. Vieweg   Sohn, Braunschweig},  (1987).


\bibitem{Hi}
{Morris W. Hirsch},   {\em Differential Topology},  \textbf {33},  Graduate Texts in Mathematics. Springer-Verlag, New York, (1976).



\bibitem{Katz}
M. ~Katz,  {\em Systolic geometry and topology}.  Math. Surveys Monographs. \textbf{137}, Amer. Math. Soc., (2007).


\bibitem{KM}
{P. B. Kronheimer and T. S. Mrowka},
{\em Scalar curvature and the Thurston norm}. 
{Mathematical Research Letters} \textbf {4} (1997), 931–937. 






\bibitem{My}
{S. ~B. Myers},  
{\em Riemannian manifolds with positive mean curvature}. 
{Duke Mathematical Journal}, \textbf {8 (2)} (1941), 401–404.






\bibitem{Nov}
{S. P. Novikov},
{\em Topology of foliations}.   {Trans. Moscow. Math. Soc},\textbf {14} (1967),  268-304.



\bibitem{PM}
{J. Palis and W.~ de Melo},  {\em Geometric theory of dynamical systems. An introduction}. 
Springer-Verlag, New York-Berlin, (1982).




\bibitem{Pet}
{ P.~Petersen},
{\em Riemannian Geometry}. Graduate Text in Mathematics, 3rd Edition,  Springer,  (2016). 



\bibitem{Pu}
{P.M. Pu},
{\em Some inequalities in certain nonorientable Riemannian manifolds}.  {Pacific J. Math.},\textbf  { 2} (1952),  55-71.



\bibitem{R}
{G. Reeb},   {\em  Sur la courbure moyenne des vari\'et\'es int\'egrales d'une \'equation de Pfaff $\omega = 0$.}  C. R. Acad. Sci. Paris, \textbf {231} (1950),  101 --102.


\bibitem{Sib}
{K. S. Sibirsky},  {\em  Introduction to Topological Dynamics}.  RIA AN MSSR, (1970) (in Russian). [English translation: Introduction to Topological Dynamics. Noordhoff, Leyden, (1975)].


\bibitem{Sul}
{D. Sullivan},
{\em A homological characterization of foliations consisting of minimal surfaces},  
{Comm. Math. Helv.},  \textbf  {54}  (1979), 218-223.

\bibitem{St}
{Daniel L. Stern},
{\em Scalar curvature and harmonic maps to $S^1$}.
{J. Differential Geom.} \textbf {122(2)} (2022),  259-269 











\bibitem{T}
{I. Tamura},
{\em Topology of foliations: an introduction}.  Translated from the 1976 Japanese edition, Translation of Mathematical Monographs \textbf{97}, American Mathematical Society, Providence, RI, (1992). 
























\bibitem{Th}
{W.P. Thurston},   {\em A norm for the homology of 3-manifolds}.  {Memoirs of the American Mathematical Society}, \textbf {59 (339)} (1986),  99 --130.



\bibitem{Th1}
{W.  P. Thurston},
{\em Three-Dimensional Geometry and Topology}.  Vol. \textbf{1}, Edited by Silvio Levi,   Princeton University Press, Princeton, New Jersey, (1992). 

\end{thebibliography}


	
	
	


\EndPaper


\end{document}